\documentclass[11pt,draft,reqno]{amsart}

\usepackage{amsmath,amssymb,amscd,amsfonts,amsthm}
\usepackage{graphics}
\usepackage{dsfont}
\usepackage[all]{xy}
\usepackage{enumerate}

{\catcode`\@=11
\gdef\n@te#1#2{\leavevmode\vadjust{%
 {\setbox\z@\hbox to\z@{\strut#1}%
  \setbox\z@\hbox{\raise\dp\strutbox\box\z@}\ht\z@=\z@\dp\z@=\z@%
  #2\box\z@}}}
\gdef\leftnote#1{\n@te{\hss#1\quad}{}}
\gdef\rightnote#1{\n@te{\quad\kern-\leftskip#1\hss}{\moveright\hsize}}
\gdef\?{\FN@\qumark}
\gdef\qumark{\ifx\next"\DN@"##1"{\leftnote{\rm##1}}\else
 \DN@{\leftnote{\rm??}}\fi{\rm??}\next@}}
\DeclareFontFamily{OT1}{wncyr}{\hyphenchar\font45 }
\DeclareFontShape{OT1}{wncyr}{m}{n}{%
   <5> <6> <7> <8> <9> gen * wncyr
   <10> <10.95> <12> <14.4> <17.28> <20.74>  <24.88>wncyr10}{}
\DeclareFontShape{OT1}{wncyr}{m}{it}{%
   <5> <6> <7> <8> <9> gen * wncyi
   <10> <10.95> <12> <14.4> <17.28> <20.74> <24.88> wncyi10}{}
\DeclareFontShape{OT1}{wncyr}{m}{sc}{%
   <5> <6> <7> <8> <9> <10> <10.95> <12> <14.4>
   <17.28> <20.74> <24.88>wncysc10}{}
\DeclareFontShape{OT1}{wncyr}{b}{n}{%
   <5> <6> <7> <8> <9> gen * wncyb
   <10> <10.95> <12> <14.4> <17.28> <20.74> <24.88>wncyb10}{}
\input cyracc.def
\def\rus{\usefont{OT1}{wncyr}{m}{n}\cyracc\fontsize{9}{11pt}\selectfont}

\DeclareMathSizes{9}{9}{7}{5} 

\theoremstyle{plain}

\newtheorem{theorem}{Theorem}

\newtheorem{proposition}{Proposition}

\newtheorem*{cornonumber}{Corollary}

\theoremstyle{definition}

\newtheorem*{definonumber}{Definition}
\newtheorem{nothing*}[theorem]{}
\newtheorem{subnothing*}[sub]{}
\newtheorem{example}{Example}

\theoremstyle{remark}

\def\bA1{{\mathbf A}\!^1}
\def\P1{{\bf P}^1}

\newcommand{\dss}{\hskip -2mm\rotatebox{68}{\raisebox{-1.8\height}{\mbox{\normalsize -\hskip .1mm-\hskip .1mm-}}}\hskip -.6mm}

\newcommand{\Psb}{{\bf P}^s}
\newcommand{\Pb}{{\bf P}^1}

\begin{document}

\title[Birational splitting]{Birational\;splitting\;and\\ algebraic\;group\;actions}

\author[Vladimir  L. Popov]{Vladimir  L. Popov${}^*$}
\address{Steklov Mathematical Institute,
Russian Academy of Sciences, Gubkina 8, Moscow\\
119991, Russia}
 \email{popovvl@mi.ras.ru}

\address{National Research University\\ Higher School of Economics\\ Myasnitskaya
20\\ Moscow 101000,\;Russia}

\thanks{
 ${}^*$\,Supported by
 grants {\rus RFFI
15-01-02158}, {\rus N{SH}--2998.2014.1}.
}

\begin{abstract}
According to the classical theorem,
every irreducible algeb\-raic variety
endowed with a nontrivial rational action of a connected linear algebraic group is birationally isomorphic
to a product of another algebraic variety and ${\bf P}^s$ with positive $s$.\;We
show that the classical proof of this theorem actually works only in characteristic $0$ and we give a characteristic free proof of it.\;To this end we prove and use a characterization of connected linear algebraic groups $G$ with the property that every rational action of $G$ on an irreducible algebraic variety is birationally equivalent to a regular action of $G$ on an affine algebraic variety.
\end{abstract}

\maketitle

1. Throughout this note $k$ stands for an algebraically closed field of arbitrary characteristic which serves
as domain of definition for each of the algebraic varieties considered below.\;Each algebraic variety is
identified with its set of $k$-rational points.\;We use freely the standard notation and conventions of
\cite{PV94}, \cite{Sp98} and  refer to \cite{Ro56}, \cite{Ro61}, \cite{Ro63},
\cite{PV94}, \cite{Po13} regarding the definitions and basic properties of rational and regular (morphic) actions of algebraic groups on algebraic varieties.\;Given a rational action of such a group $G$ on an irreducible
algebraic variety $X$, we denote by $\pi^{\ }_{G, X}\colon X\dashrightarrow X\dss G$ a rational quotient of this action; the latter means
 that $X\dss G$ and $\pi_{G, X}$ are respectively an irreducible variety and a dominant rational map such that $\pi_{G, X}^*(k(X\dss G))=k(X)^G$.

\vskip 1mm

2. Up to a change of notation and terminology, the following statement
appeared
in classical paper \cite[Thm.\;1]{Ma63}:

\begin{theorem}
\label{main}
Assume that a connected linear algebraic group $G$
acts rational\-ly and nontrivially on an irreducible algebraic variety $X$, and let $B$ be a Borel subgroup of $G$.\;Then $X$ is birationally isomorphic to
$\Psb\times X\dss B$, where $\pi^{\ }_{B, X}\colon X\dashrightarrow X\dss B$ is a rational quotient of the natural rational action of $B$ on $X$ and $0<s\leqslant \dim B$.
\end{theorem}

In \cite{Ma63} no restriction on ${\rm char}\,k$ is imposed, but actually the
brief argument given there in support of
Theorem \ref{main}
works only if ${\rm char}\,k=0$.\;We reproduce it below
in order to pinpoint where the restriction
${\rm char}\,k=0$ is implicitly used.

\begin{proof}[Argument from \cite{Ma63} supporting Theorem {\rm\ref{main}}] Since $G$ is generated by its Borel subgroups and since all Borel subgroups are conjugate to each other, $B$ acts on $X$ nontrivially.\;Since $B$ is a connected solvable linear algebraic group, there is a chain of connected subgroups
\begin{equation*}
B=B_0\supset B_1\supset \cdots \supset B_n=\{e\}
\end{equation*}
such that all $B_i$ are normal in $B$ and $\dim B_i=\dim B-i$. If $d$ is the largest index $i$ such that the action of $B_i$ on $X$ is not trivial, then let
\begin{equation}\label{quo}
\pi^{\ } _{B_d, X}\colon X\dashrightarrow X\dss B_d=:X_d
 \end{equation}
 be a rational quotient of $X$ with respect to $B_d$.\;By the cross-section theorem (\cite{Ro56}) we find that $X$ is birationally equivalent to $\Pb\times X_d$.\;The
factor group $B/B_d$ acts on $X_d$
and we can repeat the same argument.
\end{proof}

3. The assumption ${\rm char}\,k=0$ is actually implicitly used in the penulti\-ma\-te phrase of this argument.\;Indeed, it
purports the following.\;Let
be a section of $\pi^{\ }_{B_d, X}$, i.e., a rational map such that $\pi^{\ }_{B_d, X}\circ\sigma={\rm id}$.
Since $B_{d+1}$ lies in the kernel of the action of $B_d$ on $X$, this action is reduced to that of
the one-dimensional connected linear algebraic group $B_d/B_{d+1}$.\;This
action is nontrivial, hence the $B_d/B_{d+1}$-stabilizers of
points of a dense open subset of $X$ are
finite;\;in
particular, the kernel $K$ of this action is finite.\;The action of $C_d:=(B_d/B_{d+1})/K$  on $X$ is faithful, and \eqref{quo}
is its rational quotient.

Being
 a  connected one-dimensional linear algebraic group,
 $C_d$ is iso\-mor\-phic to either $k^\times$ (the multiplicative group of $k$) or $k^+$ (the additive group of $k$;
see, e.g.,\;\cite[Thm.\;3.4.9]{Sp98}.

 If it is isomorphic to $k^\times$,
 then faithfulness of its action on $X$
 implies that this action is locally free (i.e., $C_d$-stabilizers of points
 of a dense open subset of $X$ are trivial); see \cite[Lemma 2.4]{Po13}.\;Therefore, the dominant rational map
$\gamma\colon C_d\times X_d\dashrightarrow X$,
 $(c, b)\mapsto c\cdot \sigma(b)$,
 is bijective
 over a dense open subset of $X$.\;If ${\rm char}\,k=0$, the latter implies that $\gamma$ is a birational isomorphism and hence $X$ is birationally isomorphic to $\Pb\times X_d$ because the group variety of
 $C_d$ is rational.\;But if ${\rm char}\,k>0$, one can only say that
 $\gamma$ is either a birational isomorphism or purely inseparable, but not a birational isomorphism.\;The following example shows that the latter indeed may occur.

 \begin{example}\label{ex1} Let ${\rm char}\,k=p>0$.\;Consider the locally free action of $G=B=k^\times$
 on $X=
 k\setminus \{0\}$, given by $b\cdot x:=b^px$.\;We have $n=1$, $d=0$, $C_0=G$, and $X_0$
 is a point.\;Therefore, $C_0\times X_0$ is naturally identified with $G$.\;Let $\sigma$ maps $X_0$ to $1$.\;Then $\gamma^*(k(X))=k(t^p)\varsubsetneq k(t)=k(G)$, where $t$ is the standard coordinate function on $G$.\;Thus
 $\gamma$ is not a birational isomorphism.
  \end{example}

 If $C_d$ is isomorphic to $k^+$,
 the same argument works if we know that
 the action of $C_d$ on $X$ is locally free.\;If ${\rm char}\,k=0$, then local freeness indeed holds because
 in this case there are no nontrivial finite subgroups in $k^+$.\;However,
 if ${\rm char}\,k>0$,
 it may happen that the action of $C_d$ on $X$ is not locally free; therefore,
 $\gamma$ is not bijective over a dense open subset of $X$, a fortiori is not a birational isomorphism.\;The example below 
 is based on the idea going back to \cite[7.1, Example $1^\circ$]{PV94}
 and Corollary of Proposition\;\ref{Spring} below.
 \begin{example}\label{ex2} Let ${\rm char}\,k=p>0$ and let  $G=B=k^+$.\;Let $x\in k[G]$ be the standard coordinate function on $G$.\;Every homomorphism of algebraic groups $f\colon G\to G$
 defines a regular action of $G$ on $X:=k^2$ by the formula
 \begin{equation}\label{action}
 u\cdot (a, b):=(a+ub+f(u), b),\quad \mbox{where $u\in G$, $(a,b)\in X$.}
\end{equation}
From \eqref{action} we infer that an element $u\in G$ lies in the $G$-stabilizer of a point $(a, b)\in X$ if and only if $u$ is a root of the polynomial $f+bx$.\;By \cite[Lemma\;3.3.5]{Sp98}
there are nonzero elements $\alpha_1,\ldots,\alpha_s\in k$ and an increasing sequence of nonnegative integers
$n_1,\ldots, n_s$ such that
\begin{equation}\label{f}
f=\alpha_1x^{p^{n_1}}\!+\cdots+\alpha_sx^{p^{n_s}};
\end{equation}
any $\alpha_i$ and $n_j$ may occur in the right-hand side \eqref{f} for an appropriate $f$.\;One of the roots of $f$ is $0$.\;Now take
$f$ with $n_s\geqslant 1$.\;By \eqref{f} the polynomial $(f+bx)/x$ has degree $p^{n_s}-1\geqslant 1$ and does not vanish at $0$ if $b\neq 0$ and $b\neq -\alpha_1$.\;Whence if these inequalities hold, the $G$-stabilizer of $(a, b)$ is nontrivial.\;Therefore, the action is not locally free.
\end{example}

Moreover, if even the action of  $C_d$ on $X$ is locally free, and hence $\gamma$ is bijective over
a dense open subset of $X$, it may happen that  $\gamma$ is purely inseparable, but not a birational isomorphism.\;The corresponding example is similar to Example \ref{ex1}.

\begin{example}\label{ex3}  Let ${\rm char}\,k=p>0$.\;Consider the locally free action of $G=B=k^+$
 on $X=k$, given by $b\cdot x:=b^p+x$.\;Then $n=1$, $d=0$, $C_0=G$, $X_0$
 is a point, $C_0\times X_0$ is naturally identified with $G$, and if $\sigma$ maps $X_0$ to $0$, then $\gamma^*(k(X))=k(t^p)\varsubsetneq k(t)=k(G)$, where $t$ is the standard coordinate function on $G$.\;Thus
 $\gamma$ is not a birational isomorphism.
\end{example}

 \vskip 1mm

4. Below we shall give a characteristic free proof of Theorem \ref{main}.\;For this, we need the following characterization of connected linear algebraic groups $G$ with the property that every rational action of $G$ on an irreducible algebraic variety is birationally equivalent to a regular action of $G$ on an affine algebraic variety.

\begin{definonumber} We say that a linear algebraic group $G$ {\it has property} {\rm(A)} if for {\it every} rational action of $\,G$ on an irreducible algebraic variety $X$,
there
exist an irreducible {\it affine} algebraic variety $Y$ and a birational isomorphism
\begin{equation}\label{phi}
\varphi\colon X\dashrightarrow Y
\end{equation}
such that the rational action of $\,G$ on $Y$ induced by $\varphi$ is {\it regular}.
\end{definonumber}

\begin{theorem}\label{solvable}
Let $G$ be a linear algebraic group
and let $\,G^0$ be
the connected component of the identity in $G$.
\begin{enumerate}[\hskip 2.2mm\rm(i)]
\item If $\,G^0$ is solvable, then $G$ has property ${\rm (A)}$.
\item If $\,G$ is connected and has property ${\rm (A)}$, then $G$ is solvable.
\end{enumerate}
\end{theorem}
\begin{proof}
(i)$\Rightarrow$(ii):
 Let $G^0$ be solvable.\;Consider a rational action of $G$ on an irreducible algebraic variety $X$.\;By
 \cite[Thm.\;1]{Ro56}, there exists
an irredu\-cible algebraic variety $X_1$ and a birational isomorphism
\begin{equation*}
\alpha_1\colon X \dashrightarrow X_1
\end{equation*}
such that the rational action of $G$ on $X_1$ induced by $\alpha_1$ is regular.

Let $\nu\colon X_2\to X_1$ be the normalization of $X_1$.\;Then the rational action of $G$ on $X_2$ induced by the birational isomorphism
\begin{equation*}
\alpha_2:=\nu^{-1}\colon X_1 \dashrightarrow X_2
\end{equation*}
is regular, see  \cite[Thm.\;2.25]{It82}.

By \cite[Lemma\;8]{Su74}, since $X_2$ is a normal
algebraic variety, it contains a nonempty $G^{0}$-stable quasi-projective open subset $U$.\;Hence,
by \cite[Thm.\;1]{Su74}, for some positive integer $n$,
there exist a regular action of $G^0$ on the projective space ${\bf P}^n$
and a $G^0$-equivariant embedding  of the algebraic variety $U$ into ${\bf P}^n$ ,
\begin{equation*}
\iota\colon U\hookrightarrow {\bf P}^n.
\end{equation*}
We may (and shall) assume that $n$ is minimal possible with this property.
Since ${\rm Aut}\,{\bf P}^n={\rm PGL}_n$, this action of $G^0$ on ${\bf P}^n$ induces an action of
$G^0$ on the dual projective space $\check{\bf P}^n$.\;By the Borel fixed-point theorem \cite[Thm.\;6.2.6]{Sp98},
the assumption that $G^0$ is a connected solvable linear algebraic group implies
 that in $\check{\bf P}^n$ there is a fixed point of this action.\;This means that ${\bf P}^n$ contains a $G^0$-stable hyperplane $H$.\;Hence ${\bf P}^n\setminus H$ is a
$G^0$-stable affine open subset of ${\bf P}^n$.\;Therefore, the minimality assumption on $n$ implies
that
$\iota(U)\cap ({\bf P}^n\setminus H)$ is
a nonempty  $G^0$-stable   open quasiaffine  subset of $\iota(U)$.\;This proves that $X_2$ contains a nonempty
 $G^0$-stable  open quasiaffine subset $V$.\;Normality of $G^0$ in $G$ then implies that
 $g\cdot V$  for every element $g\in G$ is a nonempty
 $G^0$-stable  open quasiaffine subset of $X_2$.

 By
 \cite[Lemma 5.11 and the footnote to its proof]{BS64},
there exists a finite subgroup $F$ of $G$ that intersects every connected component of $G$.\;Put
 \begin{equation*}
 X_3:=\textstyle\bigcap_{g\in F} g\cdot V.
 \end{equation*}
 Then $X_3$ is $G^0$- and $F$-stable and, therefore, $G$-stable.\;Let
 \begin{equation*}
 \alpha_3\colon X_2\dashrightarrow X_3
 \end{equation*}
 be the birational isomorphism inverse to the identity embedding $X_3 \hookrightarrow X_2$.\;Thus we have proved that $X_3$ is a quiasiaffine algebraic variety such that the rational action of $G$ on $X_3$ induced by $\alpha_3$
 is regular.

 Finally, by \cite[Lemma\;2]{Ro61} (see also \cite[Thm.\;1.6]{PV94}) quasiaffiness of $X_3$ implies  that there exist an irreducible affine algebraic variety $X_4$ endowed with a regular action of $G$ and a $G$-equivariant birational embedding
 \begin{equation*}
 \alpha_4\colon X_3\hookrightarrow X_4.
 \end{equation*}

  This shows that we may take $Y:=X_4$ and $\varphi:=\alpha_4\circ\alpha_3\circ\alpha_2\circ \alpha_1$.

  \vskip 1mm

  (ii)$\Rightarrow$(i): Let the group $G$ be connected and has property (A).\;Assume that it is non-solvable.\;Then it contains a proper parabolic subgroup $P$; see \cite[Prop.\;6.2.5]{Sp98}.\;Let $X$ be $G/P$ endowed with the natural action of $G$.\;We have $\dim X>0$.\;Let
  $Y$ and $\varphi$ be respectively an irreducible affine algebraic variety endowed with a regular action of $G$ and a birational isomorphism \eqref{phi}, whose existence is ensured by property (A).\;Since $\varphi$ is $G$-equivariant and the action of $G$ on $X$ is transitive, $\varphi$ is a morphism.\;Therefore,
completeness and irreducibility of $X$ implies that $\varphi(X)$ is a complete $G$-stable closed irreducible subset in $Y$; see\;\cite[Prop.\;6.1.2(iii)]{Sp98}.\;Since $Y$ is affine, this yields that $\varphi(X)$ is a point; see\;\cite[Prop.\;6.1.2(vi)]{Sp98}.\;But $\dim\,\varphi(X)=\dim\,X>0$ because $\varphi$ is a birational isomorphism\,---\,a contradiction.
\end{proof}

5. We also need the following

\begin{proposition}
 \label{Spring}
 Let $G=k^+$
 and let $X$ be an irreducible
 affine algebraic variety endowed with a nontrivial regular action of $G$.\;Then there exists an irreducible affine variety $Y$ with the following properties:
 \begin{enumerate}[\hskip 2.2mm\rm(a)]
 \item there is an isomorphism $\phi$ of $G\times Y$ onto an open subvariety of $X$;
 \item there is an morphism $\psi\colon G\times Y\to G$ such that for all $a, b\in k$, $y\in Y$,
 \begin{equation*}
\psi(a+b, y)=\psi(a, y)+\psi(b, y),\quad a\cdot\phi(b, y)=\phi(\psi(a, y)+b, y).
 \end{equation*}
 \end{enumerate}
 \end{proposition}
 \begin{proof} See \cite[Prop.\;14.2.2]{Sp98} and an earlier result \cite[Lemma 1.5]{Mi78}.
 \end{proof}

\begin{cornonumber} Maintain the notation of Proposition {\rm \ref{Spring}}.\;Then the formula
\begin{equation*}
a\cdot (b, y):=(\psi(a, y)+b, y),\quad a, b\in G,\; y\in Y
\end{equation*}
defines a regular action of $\,G$ on $G\times Y$ such that
\begin{enumerate}[\hskip 2.2mm ---]
\item the natural projection
${\rm pr}_2\colon G\times Y\to Y$ is
its rational quotient;
\item  the isomorphism $\phi$
is $G$-equivariant.
\end{enumerate}
In particular, $Y$ and
$X\dss G$ are birationally isomorphic.
\end{cornonumber}

6. We now turn to a characteristic free proof of Theorem \ref{main}.

\begin{proof}[Characteristic free proof of Theorem {\rm \ref{main}}]
We retain the argument in \cite{Ma63}, except its part referring to the cross-section theorem of \cite{Ro56}
that works, as we have explained, only if ${\rm char}\,k=0$.\;This part is replaced by the following
characteristic free argument.

By Theorem \ref{solvable}, we may assume that $X$ is affine and the action of the one-dimensional connected linear algebraic group $H:=B_d/B_{d+1}$ on $X$ is nontrivial and regular.\;There are two possibilities:\;$H$ is isomorphic to either $k^+$
or $k^\times$.

Let $H$ be isomorphic to
$k^+$.\;Then by  Corollary of Proposition \ref{Spring} the variety
 $X$ is $H$-equivariantly birationally isomorphic to the variety ${\bf P}^1\times X_d$, on which $H$ acts rationally via the first factor so that the second projection ${\rm pr}_2\colon {\bf P}^1\times X_d\to X_d$ is a rational quotient of this action.

 It remains to show that the same is true if $H$ is isomorphic to $k^\times$.\;The group
 ${\mathcal X}(H)$
 of characters of $H$ (i.e., algebraic homomorphisms $H\to k^\times$)
 is isomorphic to $\bf Z$.\;Let $\chi$ be its generator.\;For every $s\in \mathbf Z$, put
 \begin{equation}\label{chara}
 k(X)_s:=\{f\in k(X)\mid h\cdot f=\chi^s(h)f\;\;\mbox{for every $h\in H$}\};
 \end{equation}
in particular, $k(X)_0=k(X)^H$.\;Since $H$ is a torus,  the $H$-module $k[X]$ is semisimple and
its isotypic decomposition has the form
\begin{equation}\label{iso}
 k[X]=\textstyle \bigoplus_{s\in \mathbf Z} k[X]_s,\quad \mbox{where $k[X]_s:=k(X)_s\cap k[X]$}
\end{equation}
(see, e.g.,\;\cite[3.2.13]{Sp98}); in particular, $k[X]_0=k[X]^H$.\;Given \eqref{iso}, every element
$f\in  k[X]$ can be uniquely written as the following sum of the nonzero summands:
\begin{equation}\label{dec}
f=f_{i_1}+\cdots+f_{i_s},\;\;\mbox{where $f_j\in k[X]_j$ for all $j$.}
\end{equation}
We call \eqref{dec} the canonical decomposition of $f$.\;Nontriviality of the action of $H$ on $X$ implies that $k[X]^H\neq k[X]$.\;Therefore, the subgroup
\begin{equation}\label{gamma}
\Gamma:=\{s\in \mathbf Z\mid k(X)_s\neq 0\}
 \end{equation}
 of $\mathbf Z$ is nonzero, i.e.,\;$\Gamma=n\mathbf Z$ for some positive integer $n$.

It follows from \eqref{gamma} that there are $a_1,\ldots, a_m\in \mathbf Z$ such that
 in canonical decomposition \eqref{dec} we have
 \begin{equation}\label{divi}
 i_1=na_1,\ldots, i_s=na_s.
 \end{equation}

 Fix a nonzero element $t\in k(X)_n$.\;From \eqref{chara},
  \eqref{dec}, and \eqref{divi}
  we deduce that
  \begin{equation}\label{inv}
  f_{i_r}/t^{a_{r}}\in k(X)^G\quad\mbox{for every $r$}.
  \end{equation}
 In turn, \eqref{dec} and \eqref{inv} yield that $f=\sum_{r=1}^st^{a_r}(f_{i_r}/t^{a_r})$ is an element of the subfield $k(X)^H\!(t)$ of $X$.\;Hence $k[X]$ lies in this subfield.\;But $k(X)$ is the field of fractions of $k[X]$ because $X$ is affine.\;This proves that
 \begin{equation}\label{t}
 k(X)=k(X)^H\!(t).
 \end{equation}

 The element $t$ is transcendental over $k(X)^H$.\;Indeed, if not, there is a relation $\sum_{i=0}^m a_it^{r_i}=0$ for some integers $0\leqslant r_0<r_1<\cdots<r_m$ and nonzero elements
 $a_i\in k(X)^H$.\;From \eqref{chara} we then deduce that
 \begin{equation*}\label{Art}
 \textstyle\sum_{i=0}^m \chi^{nr_i}(h) a_it^{r_i}=0\quad\mbox{for every element $h\in H$.}
 \end{equation*}
This contradicts Artin's theorem on independence of characters, because $k^\times$ is the subgroup   of the multiplicative group $k(X)^\times$ of $k(X)$, and therefore, we may consider the elements of $\mathcal X(H)$ as the homomorphisms $H\to k(X)^\times$.

Given that $t$ is transcendental over $k(X)^H$, we conclude from \eqref{t} that $X$ is $H$-equivariantly birationally isomorphic to the variety ${\mathbf P}^1\times X\dss H$, on which $H$ acts rationally via the first factor so that the second projection ${\rm pr}_2\colon {\bf P}^1\times X\dss H\to X\dss H$ is a rational quotient of this action.\;This completes the proof.
\end{proof}

7. Combining the given proof of Theorem \ref{main} with Rosenlicht's theorem \cite[Thm.]{Ro63} on the existence of generic geometric quotient, we obtain the following ge\-ne\-ralization of the result of \cite[Sect.\,1]{GP93} about ``trivial quotient'' of a unipotent group action on quasiaffine variety defined over a field of characteristic $0$.
\begin{theorem} Let $X$ be an irreducible algebraic variety endowed with a re\-gu\-lar action of a solvable connected linear algebraic group $G$.\;Then for the restriction of this action
on a certain $G$-stable dense open subset $\,U$ of $X$ there exist
\begin{enumerate}[\hskip 2.2mm ---]
\item the geometric quotient $\pi^{\ }_{G, U}\colon U\to U/G$;
\item an isomorphism  $\varphi\colon U\to {\bf A}^{\hskip -.5mm r, s}\times (U/G)$, where
 $${\bf A}^{\hskip -.5mm r, s}:=\{(\alpha_1,\ldots, \alpha_{r+s})\in {\bf A}^{\hskip -.5mm r+s}\mid \alpha_i\neq 0 \;\mbox{for every $i\leqslant r$}\},\;\;r\geqslant 0,\;s\geqslant 0,$$
\end{enumerate}
such that the natural projection  ${\bf A}^{\hskip -.5mm r, s}\times (U/G)\to U/G$ is the geometric quotient of the regular action of $\;G$ on ${\bf A}^{\hskip -.5mm r, s}\times (U/G)$ induced by $\varphi$.
\end{theorem}

8. {\it Acknowledgement.} I thank G. Kemper for exchange of emails that initiated this research. 
He communicated to me his characteris\-tic-free proof of
the above-mentioned 
result of  \cite{GP93} 
and drew my attention to \cite[Prop.\;14.2.2]{Sp98}, \cite[Lemma 1.5]{Mi78}. Example 2 above is a generalization of the one I first learned from him.

 \end{document}